\documentclass[12pt,reqno]{amsart}
\usepackage{graphicx}
\usepackage{psfrag}
\usepackage{pxfonts}
\usepackage{mathrsfs}
\usepackage{color}
\usepackage{amsmath,amssymb}

\numberwithin{equation}{section}

\newtheorem{theorem}{Theorem}[section]

\newtheorem{corollary}[theorem]{Corollary}
\newtheorem{proposition}[theorem]{Proposition}
\newtheorem{lemma}[theorem]{Lemma}
\newtheorem{definition}[theorem]{Definition}

\newtheorem{question}{Question}
\theoremstyle{remark}
\newtheorem{remark}[theorem]{Remark}
\newtheorem{problem}{Problem}

\begin{document}

\thanks{}

\author{F. Micena}
\address{Departamento de Matem\'atica,
  IM-UFAL Macei\'{o}-AL, Brazil.}
\email{fpmicena@gmail.com}
\author{A. Tahzibi}
\address{Departamento de Matem\'atica,
  ICMC-USP S\~{a}o Carlos-SP, Brazil.}
\email{tahzibi@icmc.usp.br}

\renewcommand{\subjclassname}{\textup{2000} Mathematics Subject Classification}

\date{\today}

\setcounter{tocdepth}{2}

\title[Regularity of foliations and Lyapunov exponents]{Regularity of foliations and Lyapunov exponents of partially hyperbolic Dynamics}
\maketitle
\begin{abstract}
In this work  we study relations between regularity of invariant foliations and Lyapunov exponents of partially hyperbolic diffeomorphisms. We suggest a new regularity condition for foliations in terms of desintegration of Lebesgue measure which can be considered as a criterium for rigidity of Lyapunov exponents.
\end{abstract}
%
%
%

%

\section{Introduction}\label{section.preliminaries}

In this paper we address the regularity of invariant foliations of partially hyperbolic dynamics and its relations to Lyapunov exponents and rigidity.
We suggest a new regularity condition (Uniform Bounded Density property) for foliations which is defined in terms of desintegration of Lebesgue measure along the leaves of the foliation. In principle it can be compared with the absolute continuity of foliations. However, for (un)stable foliations of partially hyperbolic diffeomorphisms the works of Pesin-Sinai \cite{PS}, Ledrappier \cite{L}  shed light on the subject and it turns out that for these foliations our condition impose a kind of regularity much stronger than absolute continuity.  However, we believe that exploiting this regularity condition can be considered as a geometric measure theoretical criterium for the rigidity of partially hyperbolic dynamics.


From now on, we shall consider a smooth measure $m$ (Lebesgue measure) on $\mathbb{T}^3$ and a $C^2$ diffeomorphism $f:M\to M$ preserving $m$.  $f$ is called (absolute) partially hyperbolic if there exists a
$Df$-invariant splitting of the tangent bundle $TM = E_f^s \oplus E_f^c \oplus E_f^u$ and constants
$\nu_- \leq \nu_+ < \mu_- \leq 1 \leq \mu_+ < \lambda_- \leq \lambda_+$ and $C > 0,$ satisfying
$$\frac{1}{C} \nu_-^n || v|| \leq ||Df^{n}(x) v|| \leq C \nu_+^n||v|| , \; \forall v \in E_f^s(x), $$

$$\frac{1}{C} \mu_-^n || v|| \leq ||Df^{n}(x) v|| \leq C \mu_+^n||v|| , \; \forall v \in E_f^c(x), $$

$$\frac{1}{C} \lambda_-^n || v|| \leq ||Df^{n}(x) v|| \leq C \lambda_+^n||v|| , \; \forall v \in E_f^u(x). $$

It is possible to choose a riemannian metric in $M$ that makes $C = 1$ in the above definition.  In this paper all  partially hyperbolic diffeomorphisms are defined on $\mathbb{T}^3.$ For simplicity we denote $Df(x)|E^{\sigma}_f(x)$ by $J^{\sigma}f(x), \sigma \in \{s,c,u\}.$ The distributions $E_f^s$ and $E_f^u,$ respectively  stable and  unstable bundle, are uniquely integrable to foliations $\mathcal{F}^s$ and $\mathcal{F}^u$ (See \cite{HPS}). In general case, $E^c$ is not  integrable. However for absolute partially hyperbolic diffeomorphisms on $\mathbb{T}^3$ the center bundle is integrable \cite{Br-Bu-Iv1}.

\subsection{Regularity of foliations}

Roughly speaking, a foliation is an equivalence relation on a manifold such that the equivalence classes (the leaves) are connected immersed sub manifolds. For dynamical invariant foliations, although typically the leaves enjoy a high degree of regularity they are not stacked up in a smooth fashion. To define the different regularity conditions we need foliated charts. For instance a codimension$-m$ foliation is $C^r$ if there exist a covering of the manifold by $C^r$ charts $\phi: U \rightarrow \mathbb{R}^n \times \mathbb{R}^m$ such that each plaque is sent into the hyperplane $\mathbb{R}^n \times \{\phi(p)\}.$

For a  $C^r-$partially hyperbolic diffeomorphism the invariant foliations $\mathcal{F}^s$ and $\mathcal{F}^u$ typically are at most H\"older continuous with $C^r$ leaves.
An important feature of stable and unstable foliations of partially hyperbolic diffeomorphism is their ``absolute continuity" property. In smooth ergodic theory, absolute continuity of foliations has been used by Anosov to prove the ergodicity of Anosov diffeomorphisms.  One of the weakest definitions (leafwise absolute continuity) is sufficient to prove the ergodicity of Anosov diffeomorphisms. See \cite{PVW} for other definitions and state of art of absolute continuity of foliations.

Consider $\mathcal{F}$ a foliation over $M.$  Denote by
$m$ the riemannian measure over $M$; and $\lambda_{\mathcal{F}_x}$; the riemannian  measure
over $\mathcal{F}_x$; the leaf through $x \in M.$ There is a unique desintegration $[\{m_{\mathcal{F}_x}\}]$ of $m$
along the leaves of the foliation. $[\{m_{\mathcal{F}_x}\}]$ are equivalent class of measures up to
scaling. In a foliation chart $U \subset M;$ denote $m_{\mathcal{F}_x};$ the probability measure which
comes from the Rokhlin desintegration of $m$ restricted to $U.$ In what follows we
use the unique notation $m_{\mathcal{F}_x (B)}$ to denote the desintegration of the plaque inside foliated box $B,$ which is a probability measure.
A

\begin{definition}[leafwise Absolute Continuity]\label{abs.contI}
 Let $\mathcal{F}$ be a foliation on $M.$ We say that  $ \mathcal{F}$ \textit{leafwise absolutely continuous}, if  it satisfies the following: A measurable set  $Z $ has zero Lebesgue measure if and only if for almost every $p \in M,$ the leaf  $\mathcal{F}_p$  meets  $Z$ in a  $\lambda_{\mathcal{F}_p}$ zero measure set, that is,  $\lambda_{\mathcal{F}_p} (Z) = 0,$ for almost everywhere $p \in M.$

Locally,  it is equivalent to

$$ \lambda_{\mathcal{F}_p} \sim m_{\mathcal{F}_p}, \;\; m-\mbox{almost everywhere} \;\; p \in U.$$
\end{definition}

In general setting it is not easy to understand the desintegration $[\{m_{\mathcal{F}_x}\}]$ of $m.$ In the case of leafwise absolutely continuous foliations the Radon-Nikodym derivative $\frac{dm_{\mathcal{F}_p}}{d\lambda_{\mathcal{F}_p}}$ is an interesting object to be studied.  This motivated us to introduce new regularity condition. We show that, if we assume $m_{\mathcal{F}_x}$ is ``universally proportional" to $\lambda_{\mathcal{F}_x},$ for almost everywhere $x \in M,$ independent of the size of $\mathcal{F}_x \cap U$  then many rigidity results hold.
To begin we need to work with long foliated boxes:

\begin{definition} [Long Foliated Box] \label{folbox} Let $\mathcal{F}$ be a one dimensional  foliation of $M^n.$ A set $B \subset M$ is  called a foliated box by $\mathcal{F}$ of size greater than or equal to $ R > 0,$ if:
\begin{enumerate}
\item $B$ is homeomorphic to $D^{n-1}\times (0,1)$ where $D^{n-1}$ is $(n-1)-$dimensional ball;
\item for each $x \in B,$ the length of $\mathcal{F}_x \cap B$ is greater than or equal to $ R > 0$ in the intrinsic riemannian metric of $\mathcal{F}_x.$
\end{enumerate}
\end{definition}

For any foliated box $B$ we denote by $m |B$ the normalized Lebesgue measure of $B$ and for any plaque $\mathcal{F}_x(B)$ the {\bf probability} induced Lebesgue measure on the plaque is denoted by $Leb_{\mathcal{F}_x(B)}.$ In the cases where the box is fixed, we write just $Leb_{\mathcal{F}_x}.$
\begin{definition}[Uniform Bounded Density]\label{abs.contIII}
 Let $\mathcal{F}$ be an one dimensional foliation on $M.$ We say that  $ \mathcal{F}$ has \textbf{uniform bounded density} (U.B.D) property, if there is $K > 1$ such that for every long foliated box of $\mathcal{F}$ in $M$ we have

$$\frac{1}{K} < \frac{dm_{\mathcal{F}_x}}{d Leb_{\mathcal{F}_x}} < K$$

independent of the size of  the foliated box and $x$.

\end{definition}

For example if $A$ is a linear partially hyperbolic automorphism of torus then the invariant foliations have U.B.D property.  In fact this is the case for any $f$ close to $A$ and $C^1$ conjugate to it.  If $f$ is Anosov automorphism of 3- torus we prove that in fact uniform bounded density property of unstable (or stable) foliation does imply the $C^1-$conjugacy to its linearization, see theorem \ref{thmrigidity}. We conjecture that in the context of general partially hyperbolic diffeomorphisms of $\mathbb{T}^3$ if both $\mathcal{F}^s$ and $\mathcal{F}^u$ have U.B.D property then $f$ is $C^1$ conjugate to its linearization.

Another example of  foliations with U.B.D property is the case of central foliation of ergodic partially hyperbolic diffeomorphisms on $M^3$ whenever it is absolutely continuous and the leaves are circles. Indeed as the length of central leaves are uniformly bounded (See \cite{PC} and \cite{gogolevcompact} for general statements.)the U.B.D property is equivalent to leafwise absolute continuity. A recent result of Avila-Viana-Wilkinson \cite{AVW} establishes that absolute continuity of central foliation in this setting implies $C^{\infty}$ regularity. We hope that U.B.D property of central foliations in general, may imply its differentiability.

Lyapunov exponents are important constants for measuring the assymptotic behaviour of dynamics in the tangent space level.
Let $f : M \rightarrow M $ be a measure preserving $C^1-$diffeomorphism. Then by Oseldets' Theorem for almost every $x \in M$ and any $v \in T_x(M)$ the following limit exists
$$
 \lim_{n \rightarrow \infty} \frac{1}{n} \log \|Df^n(x) v \|
$$
and is equal to one of the Lyapunov exponents of the orbit of $x.$ For a conservative partially hyperbolic diffeomorphism of $\mathbb{T}^3$ which is the main object of the study in this paper, we get a full Lebesgue measure subset $\mathcal{R}$ such that for each $x \in \mathcal{R}:$

$$
\lim_{n \rightarrow \infty} \frac{1}{n} \log \|Df^n(x) v^{\sigma} \| = \lambda^{\sigma}(f, x)
$$
where $\sigma \in \{s, c, u\}$ and $v^{\sigma} \in E^{\sigma}.$

Every diffeomorphism of the torus $f : \mathbb{T}^n \rightarrow \mathbb{T}^n$ induces an automorphism of the fundamental group and there exists a unique linear diffeomorphism $f_*$ which induces the same automorphism on $\pi_1(\mathbb{T}^n).$ $f_*$ is called the linearization of $f$ and in this paper we study the relations between Lyapunov exponents of $f$ and its linearization in the partially hyperbolic setting.

\section{Statement of Results and Questions}

First we prove that the uniform bounded density is a criterium for the rigidity of Lyapunov exponents in the context of partially hyperbolic diffeomorphisms of $\mathbb{T}^3.$

\begin{theorem} \label{thmrigidexponents} Let  $f: \mathbb{T}^3 \rightarrow \mathbb{T}^3,$ be a conservative partially hyperbolic diffeomorphism. Denote by $A = f_{\ast}$ and suppose that  stable and unstable foliations have the uniform bounded density property, then $\lambda^{\sigma}(f,\cdot) = \lambda^{\sigma}_A, \sigma \in \{s,c,u\}$ for almost everywhere $x \in \mathbb{T}^3.$
\end{theorem}
\begin{remark}
In the above theorem if we just assume the U.B.D property of one of the foliations $\mathcal{F}^{s}$ or $\mathcal{F}^{u},$ we conclude the rigidity of the corresponding Lyapunov exponent. In the above theorem the rigidity of central Lyapunov exponent is just a corollary of volume preserving property of $f.$ However, the same rigidity for central foliation also holds if we assume $\mathcal{F}^c$ has U.B.D property. As we do not have a good description for the desintegration along the central leaves, the proof for the central exponent rigidity is different from the stable and unstable foliation cases and it appears in the proof of theorem \ref{corollary1}.
\end{remark}

The above result show that U.B.D property imposes restrictions on the dynamics in the level of Lyapunov exponents.
We conjecture that U.B.D property is equivalent to $C^1$ conjugacy with linear automorphisms. For Anosov diffeomorphisms we can check this conjecture.
\begin{theorem} \label{thmrigidity} Let $f: \mathbb{T}^3 \rightarrow \mathbb{T}^3$ be a $C^2,$ conservative Anosov diffeomorphism with partially hyperbolic structure $E^{uu} \oplus E^u \oplus E^s.$ If $\mathcal{F}^u_f$ has uniform bounded density property, then $f$ is $C^{1 + \theta}$ conjugated to its linearization $A: \mathbb{T}^3 \rightarrow \mathbb{T}^3,$ for some positive $\theta,$ up to change $f$ by $f^2.$
\end{theorem}

The above theorems assume U.B.D property and conclude some rigidity of Lyapunov exponents.  We should mention that even leafwise absolute continuity imposes some restrictions on the Lyapunov exponents, as we see in the following theorem. Recall that stable and unstable foliation of any $C^2-$ partially hyperbolic diffeomorphism are leafwise absolutely continuous (\cite{Brpe}).

\begin{theorem} \label{thm1} Let $f$ be a $C^2$ conservative partially hyperbolic diffeomorphism on the $3-$torus and $A$ its linearization then

$$ \lambda^u(f,x) \leq \lambda^u(A) \; \mbox{and} \;\; \lambda^s(f,x) \geq \lambda^s(A)\; \mbox{for almost everywhere } \; x \in \mathbb{T}^3.$$
\end{theorem}

Similar to the statement of the above theorem appears in  \cite{RHRHU2} and proved in \cite{SX08}  for $f$ $C^1$-close to $A.$ In \cite{SX08}, the authors need unique homological data for the strong unstable foliation and they prove that it is the case when $f$ is closed to its linearization.

\begin{corollary}
Any conservative linear partially hyperbolic diffeomorphism is a local maximum  point for $$f \mapsto \displaystyle \int \lambda^u(f) dm.$$ Analogously any conservative linear partially hyperbolic diffeomorphism is a local minimum  point for $f \mapsto \displaystyle \int \lambda^s(f) dm.$
\end{corollary}


\begin{problem}
Classify the local maximum points of unstable Lyapunov exponent.  Are these diffeomorphisms  $C^1$ conjugated to linear?
\end{problem}
\begin{problem}
Suppose that $\lambda^{c}(f) > 0$ and $\mathcal{F}^c$ is upper leafwise absolutely continuous then $\lambda^{c}(f) \leq \lambda^{c}(A).$
\end{problem}

Another interesting issue in the setting of  partially hyperbolic diffeomorphisms is the characterization of topological type of central leaves.
 It is clear that for a general partially hyperbolic diffeomorhism (general $3-$manifolds) with one dimensional central bundle, the leaves of central foliation may be circles, line or both of them (consider suspension of an Anosov diffeomorphism of $\mathbb{T}^2$). However by Hammerlindl's result \cite{H}, central leaves of a partially hyperbolic diffeomorphism on $\mathbb{T}^3$ are homeomorphic to central leaves of its linearization and consequently all the leaves have the same topological type. A very natural question is that

\begin{question}
Suppose $f$ is volume preserving (absolute) partially hyperbolic on $\mathbb{T}^3$ and central Lyapunov exponent vanishes almost sure. Is it true that all center leaves are compact?
\end{question}
In general setting, this question has been answered negatively in \cite{PT}.
We would like to mention that by a recent result of Hammerlindl and Ures, a non-ergodic derived from Anosov diffeomorphism on $\mathbb{T}^3$, if exists, will have zero central Lyapunov exponent and non-compact central leaves. It is interesting to know whether exists example of such partially hyperbolic non-ergodic diffeomorphisms on torus.

Assuming U.B.D property  of central foliation we get the following theorem which gives an affirmative answer to the above question.

\begin{theorem} \label{corollary1}  Let  $f: \mathbb{T}^3 \rightarrow \mathbb{T}^3,$ be  a conservative  partially hyperbolic diffeomorphism. Suppose that $\mathcal{F}^c $ has the uniform bounded density property and $\lambda^c_f = 0$ for a.e. $x \in \mathbb{T}^3,$ then the center leaves are circles.
\end{theorem}

\begin{remark}
In the proof of the above theorem we show that under U.B.D condition of central foliation one concludes that $\lambda^c(f) = \lambda^c(f_*), a.e.$
\end{remark}

We do not know whether the above theorem holds just assuming leafwise absolute continuity.

\section{ Preliminaries}

In this section we review some definitions and  known results about partially hyperbolic diffeomorphisms on $\mathbb{T}^3.$
\subsection{Partially hyperbolic diffeomorphisms on $\mathbb{T}^3$}
In the rest of the preliminaries section we will recall some nice topological properties of invariant foliations of partially hyperbolic diffeomorphisms on 3-torus.
One of the key properties of the invariant foliations of partially hyperbolic diffeomorphisms in 3-torus is their quasi-isometric property. Quasi isometric foliation $W$ of $\mathbb{R}^d$ means that the leaves do not fold back on themselves much.

%

\begin{definition} \label{quasi isometric}
A foliation $W$  is quasi-isometric if there exist positive constant $Q$
such that for all $x, y$ in a common leaf of W we have
$$d_W(x, y) \leq Q^{-1} || x - y||.$$
Here $d_W$ denotes the riemannian metric on $W$ and $\|x-y\|$ is the distance on the ambient manifold of the foliation.
\end{definition}
%
%

In the partially hyperbolic case, we denote by $d^{\sigma}(\cdot, \cdot), $ the riemannian metric on $\mathcal{F}^{\sigma}, \sigma \in \{s,c,u\}.$ We define $d^c(\cdot, \cdot)$  in the dynamical coherent case. The foliation $\mathcal{F}^{\sigma}$ is called quasi isometric if its lift to the universal covering ($\mathbb{R}^3$) is quasi isometric.

The approach that we will use to prove the main theorem leads with concept of leaf conjugacy. For this we use the Hammerlindl results concerning leaf conjugacy.

\begin{theorem} [\cite{Br-Bu-Iv2}, \cite{H}]
If $f: \mathbb{T}^3 \rightarrow \mathbb{T}^3$ is partially hyperbolic diffeomorphism, then $\mathcal{\mathcal{F}}^{\sigma}, \sigma \in \{s,c,u\}$ are quasi isometric
foliations.
\end{theorem}

 \begin{proposition}[\cite{H}] \label{H1} Let $f : \mathbb{T}^3 \rightarrow \mathbb{T}^3$ be a partially hyperbolic diffeomorphism and $A: \mathbb{T}^3 \rightarrow \mathbb{T}^3$ the linearization of $f.$ Then

 $$ \lim_{||y - x || \rightarrow +\infty} \frac{y-x}{||y - x ||} = E_A^{\sigma}, \;\;  y \in \mathcal{F}^{\sigma}_x, \sigma  \in \{s, c, u\}$$

and the convergence is uniform.
\end{proposition}

\begin{proposition} [\cite{H}] \label{H2} Let $f : \mathbb{T}^3 \rightarrow \mathbb{T}^3$ be a partially hyperbolic diffeomorphism and $A: \mathbb{T}^3 \rightarrow \mathbb{T}^3$ the linearization of $f$ then for each $k \in  \mathbb{Z}$ and $C > 1$ there is an $M > 0$ such that for $x, y$,
$$||x -  y||> M \Rightarrow
\frac{1}{C} <\frac{|| f^k(x) -  f^k(y)||}
{||A^k(x) - A^k(y)||}
< C.$$

More generally, for each $k \in  \mathbb{Z}$, $C > 1$, and linear map $\pi: \mathbb{R}^d \rightarrow \mathbb{R}^d$ there is an
$M > 0$ such that for $x, y \in \mathbb{R}^d$,

$$||\pi(x - y)|| > M  \Rightarrow \frac{1}{C} <
\frac{||\pi(f^k(x) -  f^k(y))||}
{||\pi(A^k(x) - A^k(y))||}
< C.$$
\end{proposition}

\begin{theorem}[\cite{H}] \label{H3} Every partially hyperbolic diffeomorphism of the 3-torus is leaf conjugated
to its linearization, by a homeomorphism $h.$ Furthermore $h$ restricted to each center leaf is bi-Lipschitz and denoting $\widetilde{h}$ a lift of $h$ in $\mathbb{R}^3 $ one has that

$$|| \widetilde{h} - Id_{\mathbb{R}^3}||$$

is bounded.

\end{theorem}

The above theorem and propositions has the following corollaries which is useful in the rest of the paper.

 \begin{lemma} \label{linalg}
 Let $f : \mathbb{T}^3 \rightarrow \mathbb{T}^3$ be a partially hyperbolic diffeomorphism and $A: \mathbb{T}^3 \rightarrow \mathbb{T}^3$ the linearization of $f.$ For all $n \in \mathbb{Z}$ and $\epsilon > 0$ there exists $M$ such that for $x, y$ with $y \in \mathcal{F}^{\sigma}_x$ and $||x -  y||> M$ then
 $$
   (1 - \varepsilon)e^{n\lambda^{\sigma}_A } ||y -x|| \leq \|A^n(x) - A^n(y)\| \leq (1 + \varepsilon)e^{n\lambda^{\sigma}_A } ||y -x||
 $$
where $\lambda^{\sigma}$ is the Lyapunov exponent of $A$ corresponding to $E^{\sigma}$ and $\sigma \in \{s, c, u\}.$
 \end{lemma}

\begin{proof} Let us fix $\sigma$ and denote by $E_A$ the eigenspace corresponding to $\lambda^{\sigma}_A, \mu := e^{\lambda^{\sigma}_A}.$
Let  $N \in \mathbb{Z}$ and Choose $x, y \in \mathcal{F}^{\sigma}_f(x),$ such that $|| x - y || > M.$ By proposition \ref{H1}, we have

$$ \frac{x - y}{|| x - y||} = v + e_M,$$

where the vector $v = v_{E_A}$ is a unitary eigenvalue of  $A,$ in the  $E_A$ direction and  $e_M$ is a correction vector that converges to zero uniformly as  $M$ goes to  infinity.

So, considering $\mu$ the eigenvalue of $A$ in the $E_A$ direction

$$A^N \left( \frac{x - y}{|| x - y||} \right) = \mu^N v + A^N e_M = \mu^N \left(\frac{x - y}{|| x - y||} \right) -\mu^N e_M  + A^N e_M  $$

It implies that

\begin{align*} || x - y || (\mu^N - \mu^N ||e_M|| - ||A||^N || e_M||) \leq   || A^N (x - y)|| \\ \leq || x - y || (\mu^N + \mu^N ||e_M|| + ||A||^N || e_M||).
\end{align*}

Since $N$ is fixed, we can choose  $M > 0,$ such that

$$ \mu^N ||e_M|| + ||A||^N || e_M|| \leq \varepsilon \mu^N.$$

and the lemma is proved. \end{proof}

Another important fact is  that the central holonomy inside center-unstable leaves of partially hyperbolic diffeomorphisms of $\mathbb{T}^3$  is Lipschitz for distant points on the unstable leaves.

\begin{proposition} \label{distantlip}
The center holonomy $h^c$ between unstable leaves is uniformly bi-Lipschitz, for far away points in $\mathcal{F}^u.$ More precisely there is $C > 1,$ such that

$$C^{-1}\leq \frac{d^u(h^c(x), h^c(y))}{d^u(x,y)} \leq C,$$

whenever $y \in \mathcal{F}^u_x$ and $d^u(x,y)\geq 1.$

\end{proposition}
For the sake of completeness we prove the above proposition in the appendix.
Observe that the Lipschitz constant claimed in the proposition does not depend on the distance (on central leaf) between $x$ and $h^c(x).$

Finally, we recall a rigidity result in the context of Anosov diffeomorphisms. Let $f$ and $g$ be topologically conjugate Anosov diffeomorphisms, $h \circ f = g \circ h.$ We say that the periodic data of $f$ and $g$ coincide if  for every periodic point $x, f^p(x) = x,$  $Df^p(x)$ and $Dg^p(h(x))$ are conjugate operators.

\begin{theorem}[Smooth Conjugacy \cite{GoGu}]Let $f$ and $g$ be Anosov diffeomorphisms of $\mathbb{T}^3$
and $h \circ f = g \circ h,$
where $h$ is a homeomorphism homotopic to identity. Suppose that $f$ and $g$ have the same periodic data.
Assume that $f$ and $ g$ can be viewed as partially hyperbolic diffeomorphisms:
there is an $f-$invariant splitting $T\mathbb{T}^
3 = E^s_f \oplus  E^u_f \oplus E^{uu}_f$ also $T\mathbb{T}^
3 = E^s_g \oplus  E^u_g \oplus E^{uu}_g.$

Then the conjugacy $h$ is $C^{
1+\theta}$ for some $\theta> 0.$
\end{theorem}

\section{Technical Rigidity Results and proof of Theorem \ref{thmrigidexponents}}

In this section we prove some technical rigidity results for Lyapunov exponents which will be used in the proofs of the main theorems. In particular we prove Theorem \ref{thmrigidexponents}.

Let us concentrate on volume prserving partially hyperbolic diffeomorphisms of $\mathbb{T}^3.$ One important result which appears in the works of Pesin-Sinai and Ledrappier (see \cite{L} and \cite{PS}) is the exact formula for the desintegration of the Lebesgue measure along unstable manifolds (even in the Pesin theory setting): Take $\xi$ be a measurable partition subordinated to the unstable foliation. For $y \in \xi(x)$ define
\begin{equation} \label{produto}
\Delta^u(x, y) := \prod_{i=1}^{\infty} \frac{J^uf(f^{-i}(x))}{J^uf(f^{-i}(y))}.
\end{equation}
After normalizing  $\displaystyle{ \rho(y) := \frac{\Delta^u(x, y)}{L(x)}}$ where $L(x) = \displaystyle{ \int_{\xi(x)} \Delta^u(x, y) d Leb_x}$ and $\rho(\cdot)$ is the Radon-Nikodym derivative $\displaystyle{ \frac{d m_x}{d Leb_x}}.$ We emphasize that  such a clear formula for the desintegration along  a genral leafwise absolutely continuous foliation (for instance for central foliation whenever it is absolutely continuous) is not available. We use this formula in the proof of Theorem \ref{thmrigidexponents}. Here we observe some elementary properties of $\Delta^u.$
First of all note  that $C^2-$regularity of $f$ and H\"older continuous dependence of $E^u$ with the base point give us:
\begin{lemma} \label{obvio}
For any $ \epsilon> 1$ there exists $\delta > 0$ such that if $y \in W^u_{\delta, x} \subset \mathcal{F}^u_x$ then
$$
 1 -\epsilon \leq \Delta^u(x, y) \leq 1 + \epsilon.
$$
\end{lemma}
\begin{proof}
Taking logarithm, as  $\alpha-$H\"older continuity of unstable bundle and $J^uf$ implies that \begin{eqnarray*}
\log \Delta^u(x, y) &=& \sum_{i=1}^{\infty} \log J^uf(f^{-i}(x)) - \log J^uf(f^{-i}(y)\\
&\leq& \sum_{i=1}^{\infty} C_1 d(f^{-i}(x), f^{-i}(y))^{\alpha} \leq (C_1  \sum_{i=1}^{\infty}\lambda_{-}^{-i \alpha}) d^{\alpha}(x, y)
\end{eqnarray*}
where $\lambda_{-}$ comes from the definition of partial hyperbolicity. This completes the proof of the lemma.
\end{proof}

\begin{lemma}\label{definição}
Suppose that  $\mathcal{F}^u$ has bounded density property. There exists $K > 1$ such that for almost every $x \in \mathbb{T}^3$ and every $y_1, y_2 \in \mathcal{F}^u_x:$
\begin{equation} \label{productx}
K^{-1} \leq \Delta^u(y_1,y_2) \leq K.
\end{equation}
Moreover, for any $n \in \mathbb{N}:$
\begin{equation} \label{produtofinito}
K^{-2} \leq \prod_{i =0}^{n-1} \frac{J^uf (f^i(y_1))}{J^uf (f^i(y_2))} \leq K^2.
\end{equation}
\end{lemma}

\begin{proof}
By definition of uniform bounded density (\ref{abs.contIII}) it comes out that $\displaystyle{ \frac{\rho(y_2)}{\rho(y_1)}} \in [K^{-2}, K^2].$  Abusing the notations for simplicity, we substitute $K^2$ by $K$ and conclude the first claim of the lemma.

We can suppose that the points $x$ satisfying (\ref{productx}) belong to an invariant set.  So changing $x$ to $f^n(x)$ we have
\begin{equation} \label{productfnx}
K^{-1} \leq  \Delta^u(f^n(y_1), f^n(y_2)) \leq K.
\end{equation}
Dividing equation (\ref{productfnx}) by (\ref{productx}) we conclude the proof of the second claim of  lemma.
\end{proof}

In stable case, we take $f^{-1}$ and apply $(\ref{produto})$  in the $E^s_f = E^u_{f^{-1}}$ direction. Similarly for $y \in \mathcal{F}^s_x$ we define
\begin{equation} \label{sproduto}
\Delta^{s}(x, y) : = \prod_{i =0}^{\infty} \frac{J^{s}f (f^{i}(x))}{J^sf (f^{i}(y))}.
\end{equation}

%

From now on we use the notation $\Delta^{\sigma} (f^n(y_1), f^n(y_2)) = O(1)$ to denote that $\Delta^{\sigma} (f^n(y_1), f^n(y_2))$ is bounded from below and above by constants just depending on $f.$

 Now we state two technical propositions which guarantee the constancy of unstable periodic data and rigidity of Lyapunov exponents and are key to the proof of the main results.
\begin{proposition} \label{pop1}
Let $f$ be a partially hyperbolic diffeomorphism of $\mathbb{T}^3$. Suppose that for any $x$ in an invariant full measure set and any  $y_1, y_2 \in \mathcal{F}^{\sigma}_x:$  $$\Delta^{\sigma}(f^n(y_1), f^n(y_2)) = O(1),$$ then $\lambda^{\sigma}(f, x) = \lambda^{\sigma}(A)$ where $A$ is the linearization of $f$ and $\sigma \in \{s, u\}.$
\end{proposition}

\begin{proposition} \label{pop2}
Let $f$ be a partially hyperbolic diffeomorphism of $\mathbb{T}^3$. Suppose that, there exists a dense subset $D$ such that any  $y_1, y_2 \in D$ satisfy $$\Delta^{\sigma}(f^n(y_1), f^n(y_2)) = O(1)$$ then the $\sigma-$periodic data is constant for $\sigma \in \{ s,  u\},$  i.e all the periodic points have the same Lyapunov exponent in the $E^{\sigma}-$direction.
\end{proposition}
%
\subsection{Proof of Theorem \ref{thmrigidexponents}}

As we mentioned above (Lemma \ref{definição}) the strong  U.B.D.property of unstable foliation implies the desired boundedness condition and using proposition \ref{pop1} we conclude that $\lambda^u(f, x) = \lambda^u(A).$ Similarly, taking the inverse $f^{-1},$ U.B.D. property implies that $\lambda^u(f^{-1}, x) = \lambda^u(A^{-1}),$ it means $\lambda^s(f,x) = \lambda^s(A)$ and as $f$ is conservative $\lambda^c(f, x) = \lambda^c(A)$ too.

\subsection{ Proof of Proposition \ref{pop1}}
Take any $\sigma \in \{s,  u\}$ and suppose that $$ Z = \{ x \in \mathbb{T}^3| \lambda^{\sigma}(f,x) > \lambda^{\sigma}_A \}$$ has positive volume.  Let $\varepsilon > 0$ be a small number and  define
 $$A_n = \{ x \in Z| \; ||  J^{\sigma}f^m(x)|| > e^{m(\lambda^{\sigma}_A + \varepsilon )} \;\mbox{for all} \; m \geq n \}.$$

Take $n$ large enough such that

$$
m(A_n) > 0 \quad \text{ and }  \quad \frac{Q e^{n(\lambda^{\sigma}_A + \varepsilon)}}{2K^2e^{n\lambda^{\sigma}_A}}> 2.
$$
where $Q$ be as in definition (\ref{quasi isometric}) of quasi-isometric foliations (We know that stable, unstable and central foliations of partially hyperbolic diffeomorphisms in $\mathbb{T}^3$ are quasi-isometric.) and the constant $K$ is such that $$K^{-1} \leq \Delta^{\sigma} (f^n(y_1), f^n(y_2)) \leq K.$$ Similar to (\ref{produtofinito}) we get
\begin{equation} \label{produtosigmafinito}
K^{-2} \leq \prod_{i =0}^{n-1} \frac{J^{\sigma}f (f^i(y_1))}{J^{\sigma}f (f^i(y_2))} \leq K^2.
\end{equation} for any $n \in \mathbb{N}.$


Using   proposition \ref{H2} and lemma \ref{linalg}, choose $M > 0$ such that for any $y \in \mathcal{F}^{\sigma}_x$ and $d^{\sigma}(x,y) \geq M$

\begin{equation}  \label{universal1}
\frac{1}{2} < \frac{||f^nx - f^ny ||}{||A^n x - A^ny ||} < 2, \;  \; \end{equation}
\begin{equation} \label{universal2}
 \frac{1}{2}e^{n\lambda^{\sigma}_A}||x -y||< \frac{||A^n x - A^ny ||}{||x-y||} < 2e^{n\lambda^{\sigma}_A}||x -y||.
  \end{equation}


Take any regular point $x \in A_n.$ By definition we have $J^{\sigma} f^n(x) > e^{n(\lambda^{\sigma}_A + \epsilon)}$ and by (\ref{produtosigmafinito}) we get $$ J^{\sigma}f^{n} (y)  \geq \frac{1}{K^2} e^{n(\lambda^{\sigma}_A + \varepsilon)}$$ for any $y \in \mathcal{F}^{\sigma}_x.$
Now
\begin{equation}
\frac{||f^nx - f^ny ||}{|| A^nx - A^ny||} \geq \frac{Q d^{\sigma}(f^nx,f^ny)}{||A^nx - A^ny||} = \frac{Q \displaystyle \int_{\mathcal{F}^{\sigma}(B)}
|| D^{\sigma}f^n||d\lambda_{\mathcal{F}^{\sigma}_x}}{|| A^nx - A^ny||}\geq
\end{equation} \\\\
\begin{equation*}
 \geq \frac{Q e^{n(\lambda^{\sigma}_A + \varepsilon)} ||x - y||}{
2K^2e^{n\lambda^{\sigma}_A}||x - y||}\geq \frac{Q e^{n(\lambda^{\sigma}_A + \varepsilon)}}{2K^2e^{n\lambda^{\sigma}_A}}>2.
\end{equation*}
\\
which gives a contradiction.
Thus $\{ x \in \mathbb{T}^3| \lambda^{\sigma}(f,x) > \lambda^{\sigma}_A \}$ has zero volume. In the same way, considering $f^{-1},$ it comes out that  $$m(\{ x \in \mathbb{T}^3| \lambda^{\sigma}(f^{-1},x) > \lambda^{\sigma}_{A^{-1}} \}) = m(\{ x \in \mathbb{T}^3| \lambda^{\sigma}(f,x) < \lambda^{\sigma}_A \}) =0.$$

\subsection{Proof of Proposition \ref{pop2}}

\begin{proof} Suppose that there are periodic points $p,q,$ such that $\lambda^{\sigma}(p) > \lambda^{\sigma}(q).$
 Without lost of generality, suppose that $p,q$ are fixed points for $f$ and fix $\delta > 0$ such that $\lambda^{\sigma}(p) > \lambda^{\sigma}(q) + \delta.$ By hypothesis there exists $K > 1$ such that
 $$ K^{-1} \leq  \Delta^{\sigma}(f^n(x), f^n(y)) \leq K$$

 Choose small $\epsilon > 0$  such that $ \frac{1-\epsilon}{1+\epsilon} e^{\delta} > 1$ and then $n$ big enough such that
 \begin{equation} \label{quadrado}
 \left(\frac{1-\epsilon}{1+\epsilon} e^{\delta}\right)^n > K^2.
 \end{equation}

 Now take $x, y$ close to $p$ and $q$ such that
 \begin{equation} \label{mesmo}
J^{\sigma}f^k(x) \geq (1-\varepsilon)^k e^{k\lambda^{c}(p)}  \quad \text{and} \quad (1 + \varepsilon)^ke^{k\lambda^{c}(q)} \geq J^{\sigma}f^k(y) ,
\end{equation}

By \ref{quadrado} and \ref{mesmo} we get
$$
 \frac{J^{\sigma}f^n(x)}{J^{\sigma}f^n(y)} > K^2
$$
and so
 $$\Delta^{\sigma}(f^{n}(x), f^{n}(y)) =  \frac{J^{\sigma}f^n(x)}{J^{\sigma}f^n(y)}  \Delta^{\sigma}(x,y)  > K$$
which is a contradiction with the hypothesis.

%
%
%
%
%
%
%
%
%

\end{proof}

%
%
%


\section{U.B.D. Property and Rigidity for Anosov Diffeomorphisms}

In this section we prove theorem \ref{thmrigidity}.

\begin{proof}
We divide the proof into two steps where in the first step we prove constancy of the unstable periodic data and in the second one we deal with the central periodic data. The proof of the second step takes almost all of this section.

If $\mathcal{F}^u$ has the U.B.D. property, by density of $\mathcal{F}^u$ and Proposition \ref{pop2} the unstable periodic data is constant, i.e. $\lambda^u(p)= \lambda^u(q)$ for any two periodic point $p, q.$
By Anosov Closing lemma  this implies that $\lambda^u_f = \lambda^u_f(p)$ for a.e. $x,$ where $p$ is a periodic point. Again
since $\mathcal{F}^u$ has U.B.D property, by theorem \ref{thmrigidexponents} $\lambda^u_f = \lambda^u_A$ for a.e. point. Then $\lambda^u_f(p) = \lambda^u_A$ for all periodic point $p$ for $f.$

Now we claim that the central periodic data is also constant. To prove this we study Lipschitzness of the central holonomy inside center-unstable plaques. In what follows  $\Delta^c (x, y)$ is defined similar to $\Delta^u$ for the points $x, y$ on the same weak unstable leaf. It is clear that the propositions \ref{pop1} and \ref{pop2} can be applied for $\sigma = c.$

\begin{lemma}
There exists $C > 0$ such that for any $x \in \mathbb{T}^3$ and $y \in \mathcal{F}_x^u$ and any unstable leaf inside $\mathcal{F}^{cu}(x)$ the central holonomy satisfies:
$$
 C^{-1} d^u(x, y) \leq d^u(h(x), h(y)) \leq C d^u(x, y)
$$
\end{lemma}

\begin{proof}
By proposition \ref{distantlip} central holonomy is bi-Lipschitz for distant points. But the constancy of unstable Lyapunov exponent on all periodic points implies that  in fact  the holonomy is bi-Lipshitz. The proof appears in Gogolev result and we just recall it here. Take $y \in \mathcal{F}^u_x$ and $h(x), h(y)$ the images by central holonomy. With an abuse of notation we denote by $d^u$ the pseudo distance along the weak unstable leaves given by $d^u(x, y):= \displaystyle{ \int_x^y \Delta^c(x,z) d \lambda_{\mathcal{F}_x^c}(z) }$ as used in \cite{Go}.  Observe that $d^u(f(x), f(y)) = J^uf(x) d^u(x, y).$  Let $n \in \mathbb{N}$ such that $d^u(f^n(x), f^n(y)) \geq 1$ then
$$
 \frac{d^u(h(f^n(x)), h((f^n(y)))}{d^u(f^n(x), f^n(y))} = \frac{J^u(h(f^n(x)))}{J^u(f^n(x))} \frac{d^u(h(x), h(y))}{d^u(x,y)}
$$

Now using Livschitz we conclude that $J^uf$ is cohomologous to a constant function and the transfer function is continuous and consequently $\displaystyle{ \frac{D^u(h(f^n(x)))}{D^u(f^n(x))}}$ is bounded. By the choice of $n$ and proposition \ref{distantlip} the left hand side of the above equality is also bounded. So we conclude that $ \displaystyle \frac{d^u(h(x), h(y))}{d^u(x,y)}$ is bounded independent of $d^u(x,y)$ and independent of the length of central plaque jointing $x$ to $h(x).$
\end{proof}
Now we estimate the Lipschitz constant as follows. \begin{lemma}
Let $\mathcal{F}^{cu}_x$ be a leaf and $h$ the central holonomy inside it. Restricting the domain of $h$ to an small segment $[x, a]$, the Lipschitz constant of the holonomy is compared with $\Delta^c(x, y).$
\end{lemma}

\begin{corollary} \label{poslipschitz}
There exists a constant $D > 0$ such that for any $x$ and any $y \in \mathcal{F}_x^c$
$$
\Delta^c(f^n(x), f^n(y)) = O(1)
$$ and in particular all periodic points of $f$ have the same central Lyapunov exponent.
\end{corollary}

\begin{proof} (of lemma)
Consider $y = h(x)$ (see figure $1.$) and take $\varepsilon > 0$  such that $\varepsilon \ll \lambda_{\mathcal{F}^u}([x,a]).$ Build rectangles $R_1, R_2$ such that $\lambda_{\mathcal{F}^c_z}(\mathcal{F}^c_z \cap R_i) = \varepsilon, i = 1,2.$ The Lebesgue measure of center-unstable leaf $\lambda_{\mathcal{F}^{cu}}$ desintegrate to conditional measures on central leaves which are absolutely continuous with respect to $\lambda_{\mathcal{F}^c(.)}$ and $\hat{\lambda}_{\mathcal{F}^{cu}}$ is the quotient measure:

\begin{equation}\label{rokhlin}
\lambda_{\mathcal{F}^{cu}} (R_i) = \int_{[x, a]_u} d \hat{\lambda}_{\mathcal{F}^{cu}}(z) \int_{\mathcal{F}^c_z\cap R_i} \rho_z(t_i) d\lambda_{\mathcal{F}^c_z} (t_i), \quad i =1, 2.
\end{equation}

\begin{center} \label{lipo}
\begin{figure}[ht]
\centering{

\psfrag{R1}{\footnotesize$R_1$}
\psfrag{R2}{\footnotesize$R_2$}
\psfrag{x}{\footnotesize$x$}
\psfrag{a}{\footnotesize$a$}
\psfrag{y}{\footnotesize$y$}
\psfrag{ha}{\footnotesize$h(a)$}
\psfrag{z}{\footnotesize$z$}
\psfrag{h(z)}{\footnotesize$h(z)$}
\psfrag{t1}{\footnotesize$t_1$}
\psfrag{t2}{\footnotesize$t_2$}

\includegraphics[scale=0.4]{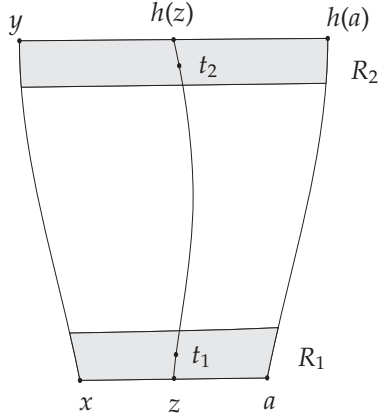}
}
\caption{Central holonomy inside center-unstable leaf}
\end{figure}
\end{center}

We know that (\cite{PS}) for $z_i \in \mathcal{F}^c_z,$
\begin{equation} \label{densities}
\frac{\rho_z(t_1)}{\rho_z(z)} = \Delta^c(z, t_1)  \quad \text{and} \quad \frac{\rho_z(t_2)}{\rho_z(z)} = \Delta^c(z, t_2) = \Delta^c(z, h(z)) \Delta^c(h(z), t_2).
\end{equation}

As $z, h(z)$ belongs respectively to the local unstable leaf of $x$ and $y$ and the center holonomies are uniform bi-Lispchitz, we have \begin{equation}  \label{unstable}
\Delta^c (z, h(z)) \sim \Delta^c(x, y).
\end{equation}

Substituting,
$$
\int_{\mathcal{F}^c_z\cap R_i} \rho_z(t) d\lambda_{\mathcal{F}^c_z} (t)= \int_{\mathcal{F}^c_z\cap R_i} \rho_z(z) \Delta^c(z, t) d\lambda_{\mathcal{F}^c_z} (t) $$
in \ref{rokhlin} and using \ref{densities}, \ref{unstable} we have
$$
 \frac{\lambda_{\mathcal{F}^{cu}} (R_2)}{\lambda_{\mathcal{F}^{cu}} (R_1)} \sim  \Delta^c(x, y)  \times\frac{ \displaystyle \int_{[x, a]_u} d \hat{\lambda}_{\mathcal{F}^{cu}(z)} \int_{\mathcal{F}^c_{z}\cap R_2}  \Delta^c(h(z), t_2) d\lambda_{\mathcal{F}^c_{z}} (t_2) }{  \displaystyle \displaystyle \int_{[x, a]_u} d \hat{\lambda}_{\mathcal{F}^{cu}(z)} \int_{\mathcal{F}^c_{z}\cap R_1}  \Delta^c(z, t_1) d\lambda_{\mathcal{F}^c_{z}} (t_1) }
$$
Finally observe that $t_1, t_2$ belong to local central leaf of $z_1$ and $h(z_1).$ As central leaf is the weak unstable foliation we again use lemma \ref{obvio} to conclude that both $\Delta^c(z, t_1)$ and $\Delta (h(z), t_2)$ are close to one.
This implies that
$$
\frac{\lambda_{\mathcal{F}^{cu}} (R_2)}{\lambda_{\mathcal{F}^{cu}} (R_1)} \sim \Delta^c(x, y).
$$
As $\epsilon$ is small the Lebesgue measure of $R_i$ is compared with the product of base and length and consequently the above relation shows that $\Delta^c (x, y)$ is compared with the Lipschitz constant of the central holonomy.

\end{proof}

\begin{proof} (of corollary)
In the above lemma we have proved that $\Delta^c(x, y)$ is comparable with the Lipschitz constant of the central holonomy which is the same for all center-unstable plaques in the manifold. Just substitute $x$ and $y$ by $f^n(x)$ and $f^n(y).$ Now, use Proposition \ref{pop2} and density of center manifold to conclude that $\lambda^c(p)= \lambda^c(q)$ for any two periodic points $p, q.$
\end{proof}

Now as the central periodic data is constant by Anosov closing lemma we conclude that $\lambda^c(f, x) = \lambda^c(p)$ for any periodic point $p.$ Using  corollary \ref{poslipschitz}  and Proposition \ref{pop1} we conclude that $\lambda^c(f, x) = \lambda^c(A)$.

Since $f$ and $A$ are conservative diffeos

$$ 0 = \lambda^u_f(p) + \lambda^c_f(p) + \lambda^s_f(p) = \lambda^u_A + \lambda^c_A + \lambda^s_A \Rightarrow \lambda^s_f(p) =\lambda^s_A$$
for every periodic point $p$ of $f.$ So, $f$ and $A$ have the same periodic data, up to change of $f$ by $f^2.$
Using \cite{GoGu}, $f$ is $C^{1 + \theta}$ conjugated to $A,$ for some positive $\theta.$ \end{proof}

\begin{corollary} Let $f,A$ like in previous theorem, then are equivalents:

\begin{enumerate}
\item $\mathcal{F}^u$ has U.B.D. property,
\item $f$ is $C^{1 + \theta}$ conjugate to $A,$ up to change $f$ by $f^2,$
\item $f$ is differentialy conjugated to $A,$ up to change $f$ by $f^2,$
\item $f$ and $A$ have the same periodic data, up to change $f$ by $f^2.$
\end{enumerate}

\end{corollary}

\section{Local maxima for Lyapunov exponents}

\subsection{Proof of theorem \ref{thm1}}

\begin{proof}
We prove the statement on $\lambda^u(f, \cdot).$ Suppose by contradiction that there is a  positive set $Z \in \mathbb{T}^3,$ such that, for every $x \in Z$ we have  $\lambda^u(f,x) > \lambda^u(A).$
Since $f$ is  $C^2,$ the unstable foliation $\mathcal{F}^u$ for $f$ is upper absolutely continuous, in particular there is a positive set $B$ such that for every point $x \in B$ we have

\begin{equation}
 \lambda_{\mathcal{F}^u_x}(\mathcal{F}^u_x \cap Z) > 0.
\label{1}
\end{equation}
Choose a $p \in B$ satisfying (\ref{1}) and $\varepsilon > 0$ a small number. Now consider a segment $[x,y]_u \subset \mathcal{F}^u_p $ satisfying
 $\lambda_{\mathcal{F}^u_p}([x,y]_u \cap Z) > 0$ such that the length of $[x,y]_u$ is bigger than $M$ as required in lemma \ref{linalg} and proposition \ref{H2}. We can choose $M$ such that

$$||Ax - Ay|| < (1 + \varepsilon)e^{\lambda^u_A } ||y -x|| $$

and

$$\frac{|| fx - fy|| }{ ||Ax - Ay||} < 1 + \varepsilon. $$

 whenever $d^u(x, y) \geq M.$ The above equation implies that $$ || fx - fy|| < (1+ \varepsilon)^2 e^{\lambda^u_A} || y - x||.$$

Inductively, we assume that for  $n \geq 1$ we have

\begin{equation}
|| f^nx - f^ny||< (1+\varepsilon)^{2n} e^{n \lambda^u_A }|| y - x||. \label{induction}
\end{equation}

 Since $f$ expands uniformly  on the $u-$direction we have $d^u(f^nx,f^ny) > M,$ it leads

\begin{eqnarray*}
|| f(f^nx) - f(f^ny)|| &<& (1+\varepsilon)|| A(f^nx) - A(f^ny)|| \\  &<& (1 + \varepsilon)^2 e^{\lambda^u_A} || f^nx - f^n y||\\ &<&
 (1+\varepsilon)^{2(n+1)} e^{(n+1)\lambda^u_A}.
\end{eqnarray*}

For each $n > 0,$ let $A_n \subset Z$ be the following set

$$A_n = \{ x \in Z \colon\;\; \|D^uf^k \| > (1+2\varepsilon)^{2k} e^{k\lambda^u_A} \;\; \mbox{for any} \;\; k \geq n\}. $$
We have $m(Z) > 0$ and $A_n \uparrow Z.$
Consider a big $n$ and  $\alpha_n > 0$ such that $ Leb_{\mathcal{F}^u_p} ([x,y]_u \cap A_n) = \alpha_n Leb_{\mathcal{F}^u_p} ([x,y]_u).$
Note that when $n $ increases to infinity the proportion $\alpha_n $ converges to $$ Leb_{\mathcal{F}^u_p} ([x,y]_u \cap Z) .$$ We can assume with lost generality $\alpha_n > \alpha_0 > 0$ for any $n > 1.$ Then

\begin{eqnarray}
||f^nx - f^ny || &>&  Q \displaystyle\int_{[x,y]_u \cap A_n} ||Df^n(z)|| d\lambda_{\mathcal{F}^u_p}(z) >  \\ &>&
 Q (1+ 2\varepsilon)^{2n}
e^{n \lambda_A^u } \lambda_{\mathcal{F}^u_p} ([x,y]_u \cap A_n) \\ &>& \alpha_0 Q (1 + 2\varepsilon)^{2n} e^{n\lambda^u_A} \|x-y\|. \label{conclusion}
\end{eqnarray}
The inequalities $(\ref{induction})$ and $(\ref{conclusion})$ give a contradiction. We conclude that $\lambda^u(f,x) \leq \lambda^u(A),$ for almost everywhere $x \in \mathbb{T}^3.$
Considering the inverse $f^{-1} $ we conclude that
$ \lambda^s(A) \leq  \lambda^s(f,x)$
for  almost every $x \in \mathbb{T}^3.$\end{proof}

The above arguments also can be used to prove similar statements for absolutely continuous central foliations of Anosov diffeomorphisms.

\begin{theorem}\label{thm2} Consider $f: \mathbb{T}^d \rightarrow \mathbb{T}^d,  d \geq 1,$ a $C^2$ Anosov, conservative with partially hyperbolic structure diffeomorphism,  such that $\dim E^c_f = 1.$ Let $A: \mathbb{T}^d \rightarrow \mathbb{T}^d, $ the linearization $f.$  Suppose that $\mathcal{F}^c$ is  quasi isometric, upper absolutely continuous foliation and $\lambda^c(f,\cdot) > 0,$ for a.e. $x\in \mathbb{T}^d, $ then $\lambda^c(f, \cdot) \leq \lambda^c_A,$ a.e. $x\in \mathbb{T}^d $
\end{theorem}

\begin{proof} By \cite{H}  we have $A$ is partially hyperbolic and $\dim E^c_f = \dim E^c_A.$ Since $f$ is Anosov with partially hyperbolic structure, then by propositions $\ref{H1}$ and $\ref{H2},$ we have that $A$ is Anosov and $\lambda^c_A > 0.$ for a.e.  Furthermore there is $\mu > 1,$ such that $|| D^cf(x)|| > \mu,$ for any $x \in \mathbb{T}^3.$ In the other words, $f$ has uniform expansion in the center direction. Since $\mathcal{F}^c_f$ is   quasi isometric, we can apply the same argument of the previous theorem, and we conclude $\lambda^c(f, \cdot) \leq \lambda^c_A,$ a.e. $x\in \mathbb{T}^d. $ \end{proof}

\section{Topology of central leaves}

To prove  theorem \ref{corollary1} we show that U.B.D property  of central foliation implies that for almost every $x$ we have $\lambda^c(x) = \lambda^c_A.$ Consequently $A$ has compact central leaves and by leaf conjugacy between $f$ and $A$ the same is true for $f.$

The idea of the proof is similar to that of theorem \ref{thmrigidexponents}. However a main difficulty here is that we do not know have a formula for the density of the desintegration along central foliation. Let $$Z = \{  x \in \mathbb{T}^3 | \lambda^c(f, x) > \lambda^c_A\}$$
and  $\varepsilon > 0$ be a small number. Define
 $$A_n = \{ x \in Z| \; ||  Df^m(x)|E^{c}_x || > e^{m(\lambda^{c}_A + \varepsilon )} \;\mbox{for all} \; m \geq n \}.$$

Take $n$ large enough such that

\begin{equation} \label{largen}
m(A_n) > \alpha  \quad \text{ and }  \quad \frac{Q\alpha e^{n(\lambda^{\sigma}_A + \varepsilon)}}{4Ke^{n\lambda^{\sigma}_A}}> 2.
\end{equation}
for some positive $\alpha> 0$ which will be fixed.
We choose $M$ satisfying \ref{universal1} and \ref{universal2} for $\sigma = c$.

Now the idea is to find a central plaque $\mathcal{F}^c_x$ of size $M$ such that $Leb_x(A_n) \geq \alpha /2K.$ Of course, if we could provide a measurable partition of $M$ into plaques of size $M$ by Rokhlin decomposition we would get a plaque such that $m_x(A_n) \geq \alpha$ and by definition $Leb_x(A_n) \geq \alpha/K.$ As we ignore the existence of such partition we construct a partition covering of a large measurable subset of $\mathbb{T}^3$ in the next subsection.

Let us complete the proof of the theorem assuming the existence of such plaque. The idea is to get the same contradiction as in the proof of theorem \ref{thmrigidexponents}. More precisely, we get similar to $(\ref{conclusion})$

\begin{equation}
\frac{||f^nx - f^ny ||}{|| A^nx - A^ny||} \geq \frac{Q \displaystyle \int_{\mathcal{F}^{c}_x(B)}
|| Df^n|E^{c}||d\lambda_{\mathcal{F}^c_x}}{|| A^nx - A^ny||} \geq
\end{equation}
\begin{equation*}
 \geq\frac{Q \alpha e^{n(\lambda^{c}_A + \varepsilon)} ||x - y||}{
4Ke^{n\lambda^{\sigma}_A}||x - y||}\geq \frac{Q \alpha e^{n(\lambda^{\sigma}_A + \varepsilon)}}{4Ke^{n\lambda^{\sigma}_A}}>2.
\end{equation*}

To find such a plaque we need to construct a measurable partition by plaques as follows.
\subsubsection{\bf Measurable partition by long plaques}

  It is more convenient to work in the universal covering $ \pi: \mathbb{R}^3 \rightarrow \mathbb{T}^3.$
We lift the foliations to $\mathbb{R}^3$ and use the same notations $\mathcal{F}^{\sigma}_x$ for the leaf passing through $x$ in $\mathbb{R}^3$.  First we recall a nice property of central foliation proved in Hammerlind:

\begin{proposition}
There is a constant $R_c$ such that for all $x \in \mathbb{R}^3$, $\mathcal{F}_{x}^c \in U_{R_c}(\mathcal{A}_{x}^{c})$ where $U_{R_c} (\mathcal{A}_{x}^{c})$ denote the neighbourhood of radius $R_c$ around the central leaf of $A$ through $x.$
\end{proposition}

For  $M$ large enough, we take $D$ a ball centered at $ O \in \mathbb{R}^3$ of radius $M$, transverse to $\mathcal{F}^c$ and in the $su-$leaf of $A$. Now saturate $D$ by central plaques of size $M$ and let
$
 \hat{D} := \bigcup_{z \in D} \mathcal{F}_{z, M}^c.
$

\begin{lemma}
If $M$ is large enough there exists a plaque $\mathcal{F}_{z, M}^c$ such that $$Leb_{z} (\pi^{-1}(A_n)) \geq \alpha /2K.$$
\end{lemma}

\begin{proof}
Recall that $m(A_n) \geq \alpha.$ As $M$ is large, $\hat{D}$ will include a large number $N(M)$ of fundamental domains (cubes) where $\pi$ is invertible. That is $C_i \subset \hat{D}$ where $C_i$ are unitary cubes for $i = 1, \cdots, N(M)$. However $\hat{D}$ may intersect partially $\tilde{N}(M)$ other fundamental cubes, i.e $\hat{D} \cap \tilde{C_i} \neq \emptyset$ but $ \tilde{C_i} \notin \hat{D}$ for $i = 1, \cdots, \tilde{N}(M).$
By the above proposition we claim that $$\lim_{M \rightarrow \infty} \frac{\tilde{N}(M)}{N(M)} = 0.$$

So for large enough $M$ we have $$  m (\pi^{-1}(A)_n) \cap \hat{D}) \geq  \frac{\alpha N(M)}{ N(M) + \tilde{N}(M)} \geq \alpha/2.$$
Now we desintegrate along the plaques in $\hat{D}$ by Rokhlin we get plaques such that $m_z(\pi^{-1}(A_n)) \geq \alpha/2$ and by definition \ref{abs.contIII} of uniform bounded density property, it yields that $Leb_{z} (\pi^{-1}(A_n)) \geq \alpha /2K.$

\end{proof}

%
%
%
%
%
%
%
%
%

\section{Appendix}

Here we prove the proposition \ref{distantlip}.

\begin{definition}
Let $x,y$ different points such that $y \in \mathcal{F}^u_x.$  Define $Q_{\infty}(x,y) \subset \mathcal{F}^{cu}_x $ being the strip  limited by center leaves $\mathcal{F}^c_x$ and $\mathcal{F}^c_y,$

\end{definition}

\begin{lemma}[Geometric Control] \label{trigo1} Let  $f: \mathbb{T}^3 \rightarrow \mathbb{T}^3$ be a  partially hyperbolic diffeomorphism. Then  there exists a positive constants $C > 1$ such that  for any strip $Q_{\infty}(x,y),$ we have

 $$C^{-1} \leq d^u(z, h_u(z)) \leq C,$$

 where $z \in \mathcal{F}^c_x$ and $h_u(z) = \mathcal{F}^u_z \cap \mathcal{F}^c_y,$ the image  of $z$ by unstable holonomy $h_u:\mathcal{ F}^c_x \rightarrow \mathcal{F}^c_y.$

\end{lemma}

\begin{proof} First fix $x,y  \in \mathbb{T}^3$ such that $y \in \mathcal{F}^u_x$ and $d^u(x,y) = 1.$  Consider the foliations in the covering space $\mathbb{R}^3$ and take the strip $Q_{\infty}(x,y).$  Let $h: \mathbb{R}^3 \rightarrow \mathbb{R}^3,$ be as in theorem \ref{H3}, leading leaves of  $\mathcal{F}^c$ in leaves of $\mathcal{F}^c_A$ where $A,$ is the linearization of $f.$ Call $L_x = h(\mathcal{F}^c_x)$ and $L_y = h(\mathcal{F}^c_y).$

Suppose that  there is a sequence $x_n \in \mathcal{F}^c_x,$  such that $d^u(x_n, y_n) \rightarrow 0,$ where  $y_n = h_u(x_n).$ By continuity of  $h,$ we have

$$ || h(x_n) - h(y_n) || \rightarrow 0 \Rightarrow d(L_x, L_y) = 0,$$

that is the contradiction and it proves that the distance of $z$ and $h_u(z)$ is bounded below, $d^u(z, h_u(z)) \geq \eta.$

Now we prove that the distances are bounded above uniformly. Suppose that there is a sequence $x_n \in \mathcal{F}^c_f(x),$ such that  $d^u(x_n , y_n) \rightarrow +\infty,$ where $y_n = h_u(x_n).$ Construct $z_n \in \mathcal{F}^c_y,$ as follows.

Let  $u_n = h(x_n),$ and consider $v_n$ a point on $L_y,$ satisfying

$$d(u_n , v_n) = d(L_x, L_y).$$

Finally, consider $z_n = h^{-1}(v_n),$ we have

$$||x_n - z_n || = ||  (x_n - h(x_n)) + (u_n - v_n) + (h(z_n) - z_n)|| \leq d + 2 \delta,$$

such that  $d = dist(L_x, L_y)$ e $\delta = || h -id_{\mathbb{R}^3}||.$

Consider the  triangles  $\bigtriangleup_n,$ with vertices on  $x_n, z_n, y_n.$ Consider on each vertice the respective angles $\widehat{x_n}, \widehat{y_n}, \widehat{z_n}.$

\begin{center}
\begin{figure}[ht]

\centering{

\psfrag{Fcx}{\footnotesize$\mathcal{F}^c_x$}

\psfrag{Fcu}{\footnotesize$\mathcal{F}^c_y$}

\psfrag{Fu}{\footnotesize$\mathcal{F}^u_{x_n}$}

\psfrag{xn}{\footnotesize$x_n$}

\psfrag{yn}{\footnotesize$y_n$}

\psfrag{zn}{\footnotesize$z_n$}

\includegraphics[scale=0.33]{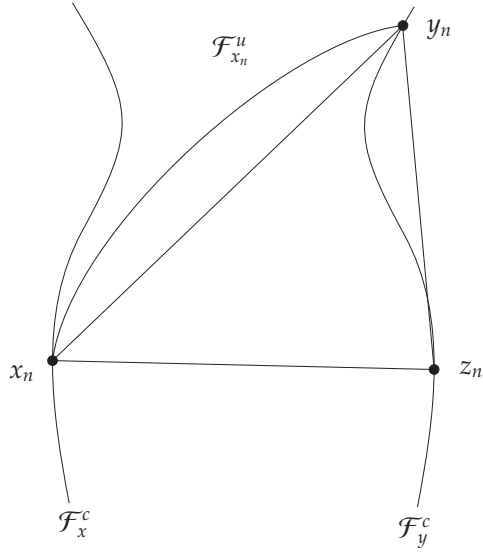}
\caption{Triangles $\Delta_n.$} }

\end{figure}
\end{center}

Note that $ || x_n - y_n|| \rightarrow + \infty$ by  quasi isometry of the foliation $\mathcal{F}^u.$

Triangular inequality implies
$$ ||  x_n - z_n|| + || y_n - z_n || \geq ||x_n - y_n ||$$
and  since $|| x_n - z_n||$ is  bounded consequently
$$ || y_n - z_n || \rightarrow +\infty.$$

  We have $|| x_n - y_n||, || y_n - z_n || \rightarrow +\infty,$ by proposition \ref{H1} we have  $\widehat{y_n}$ converges to $\theta = \angle (E^c_A, E^u_A) > 0.$
%
%
%
%
This gets a contradiction with $\|x_n - z_n\|$ bounded. Then, there exists  $K > 0,$ such that $d^u(z, h_u(z)) \leq K$ for any  $z \in \mathcal{F}^c_x.$ So we conclude that for any unstables segments $T$ and $T'$ connecting $\mathcal{F}^c_x$ and $\mathcal{F}^c_y$ we have

\begin{equation}
 \frac{\eta }{K} \leq \frac{||T ||}{||T' ||} \leq \frac{K}{\eta}, \label{lips1}
 \end{equation}

 where $||T||$ denotes the length of a unstable segment $T.$

 Consider now  another unstable segment $[a,b]_u,$ such that $d^u(a, b) = 1.$ Define $Q_{\infty}(a,b)$ the strip in $\mathcal{F}^{cu}_a$ limited by center leaves $\mathcal{F}^c_a$ and $\mathcal{F}^c_b.$ By density of foliations $\mathcal{F}^{cu}$ and $\mathcal{F}^c,$ and uniform convergence of holonomies $h_u$ and $h_c,$ we have from this and $(\ref{lips1})$ that

 \begin{equation}
  \frac{\eta }{2K} \leq \frac{||T ||}{||T' ||} \leq \frac{2K}{\eta} \label{lips}
 \end{equation}

 for any unstables segments $T$ and $T'$ connecting $\mathcal{F}^c_a$ and $\mathcal{F}^c_b,$ in particular when $T' =[a,b]_u, \; || T'|| = 1,  $ we have

 $$\frac{\eta }{2K} \leq ||T || \leq \frac{2K}{\eta}$$

 independent of the strip $Q_{\infty}(a,b).$
 \end{proof}

Now we complete the proof of proposition \ref{distantlip}.
Consider the strip $Q = Q_{\infty}(x,y),$ and unstables segments $T,T'$ connecting $\mathcal{F}^c_x$ and $\mathcal{F}^c_y.$ Suppose that $T = [a,b]_u ,$ such that $d^u(a,b) > 1.$ We claim that  there is a universal constant $C > 1$ independent of $Q,x,y$ such that

$$ C^{-1} \leq \frac{||T' ||}{||T||} \leq C.$$

To prove the claim,
On $[a,b]_u$ consider the partition $[a,b]_u = [a_0,a_1]_u \cup [a_1,a_2]_u \cup \ldots \cup [a_{n-1},a_n]_u,$ such that $a_0 = a, a_n = b.$ Choose the points $a_i, i=1, \ldots, n-1$ satisfying $d^{u} (a_{i-1}, a_i ) = 1$ and $d^u(a_{n-1}, b) \leq 1.$
By the previous lemma, we have the length of $T'$ is at most $\frac{2nK}{\eta}$ and at least $\frac{(n-1)\eta}{2K}.$ Note that, since we supposing here $d^u(a,b) > 1$ we have $n \geq 2.$ Thus

$$\frac{\eta}{4K} \leq \frac{(n-1)\eta}{2nK}\leq \frac{|| T '||}{|| T || } \leq \frac{2nK}{(n-1)\eta} \leq \frac{4K}{\eta}$$

Take $C = \frac{4K}{\eta}$ being the bi-Lipschitz constant for large $u-$scale.

\end{document}